\documentclass[10pt,a4,oneside]{amsart}
\usepackage{amsmath}
\usepackage{latexsym}
\usepackage{amssymb}
\usepackage{dsfont}
\usepackage{graphicx}
\usepackage{amsthm}
\usepackage{paralist}
\usepackage{anysize}
\usepackage{graphics}
\usepackage[colorlinks=true]{hyperref}
\usepackage[dvipsnames,usenames]{color}

\usepackage[hmargin=3cm, vmargin=4cm]{geometry}

\allowdisplaybreaks[1]

\makeatletter

\def\theequation{\thesection.\@arabic\c@equation}
\makeatother
\newtheorem{theorem}{Theorem}[section]

\newtheorem{proposition}[theorem]{Proposition}

\begin{document}
\title[A Paneitz-type  problem in pierced domains]{A Paneitz-type  problem in pierced domains}
\author[S. Alarc\'on]{S. Alarc\'on$^{\dag}$}
\author[A. Pistoia]{A. Pistoia$^{\ddag}$}
\address[$\dag$]{Departamento de  Matem\'atica,
Universidad T\'ecnica Federico Santa Mar\'ia, Casilla 110-V,
Valpara\'iso, Chile.}
\email{salomon.alarcon@usm.cl}
\address[$\ddag$]{Dipartimento SBAI, La Sapienza Universit\`a di Roma, via A. Scarpa 16, 00161 Roma, Italy.}
\email{pistoia@dmmm.uniroma1.it}
\date{\today}
\maketitle
\centerline{\em "Dedicated to Professor Takashi Suzuki on the occasion of his sixtieth
birthday"}
\begin{abstract}
We study the critical problem
\begin{equation}
\left\{
\begin{array}{ll}
\Delta ^{2}u=u^{\frac{N+4}{N-4} } & \mbox{ in }\Omega\setminus \overline{B(\xi_0,\varepsilon) },\medskip\\
u>0&\mbox{ in }\Omega\setminus \overline{B(\xi_0,\varepsilon) },\medskip\\
u=\Delta u=0 & \mbox{ on }\partial (\Omega \setminus \overline{B(\xi_0,\varepsilon) }),
\end{array}
\right. \tag{P$_\varepsilon$}
\end{equation}
where $\Omega$ is an open bounded domain in $\mathbb{R}^N$, $N\ge5$,  $\xi_0\in\Omega$ and  $B(\xi_0,\varepsilon)$ is the ball centered at $\xi_0$ with radius $\varepsilon>0$   small enough. We construct solutions of  (P$_\varepsilon$) blowing-up at the center of the hole  as the size of the hole goes to zero.
\end{abstract}

\setcounter{equation}{0}
\section{Introduction}
This paper deals with  the following fourth order problem involving the bi-Laplacian operator
\begin{equation}\label{P}
\left\{
\begin{array}{ll}
\Delta ^{2}u=u^{N+4\over N-4  } & \mbox{ in }\Omega ,\medskip\\
u>0&\mbox{ in }\Omega ,\medskip\\
u=\Delta u=0 & \mbox{ on }\partial  \Omega  ,
\end{array}
\right.
\end{equation}
where $\Omega$ is an open bounded domain in $\mathbb{R}^N$ and $N\ge5$. The exponent $N+4\over N-4$ is the critical Sobolev exponent for the  embedding
$H^2(\Omega)\cap H^1_0(\Omega)\hookrightarrow L^{p+1}(\Omega)$.
 The interest in this equation grew up from its resemblance to some geometric equations involving Paneitz operator and widely studied in the last years
 by Branson-Chang-Yang \cite{BCY}, Chang \cite{C}, Chang-Gurski-Yang \cite{CGY} and Chang-Yang \cite{CY}.
 Solvability of \eqref{P} is a delicate issue and depends strongly on the geometry of the domain $\Omega.$ Indeed,   Van der Vorst \cite{VanderVorst} proved that (\ref{P}) does not admit any positive solutions when $\Omega$ is star-shaped while Ebobisse and Ahmedou \cite{EbobisseAhmedou} showed that (\ref{P})  possesses a solution provided that some homology group of   $\Omega$ is nontrivial. This topological assumption is  sufficient  but not necessary,   as Gazzola, Grunau and Squassina \cite{GazzolaGrunauSquassina} pointed out, by showing examples
 of contractible domains on which a solution to \eqref{P}   exists.

 In this paper, we are concerned with the case when the domain  has
a circular hole which shrinks to a point, i.e.
\begin{equation}\label{P2}
\left\{
\begin{array}{ll}
\Delta ^{2}u=u^{N+4\over N-4  } & \mbox{ in }\Omega_\varepsilon:=\Omega\setminus \overline{B(\xi_0,\varepsilon) } ,\medskip\\
u>0&\mbox{ in }\Omega_\varepsilon  ,\medskip\\
u=\Delta u=0 & \mbox{ on }\partial  \Omega_\varepsilon  ,
\end{array}
\right.
\end{equation}
  where $\xi_0\in \Omega$ and $\varepsilon$ is a small positive parameter. The domain $\Omega_\varepsilon $ is topologically non trivial, so  Ebobisse and Ahmedou \cite{EbobisseAhmedou} ensures the existence of a solution for any $\varepsilon<{\mathrm dist}(\xi_0,\partial\Omega).$  A natural question arises: {\it which is the asymptotic profile of this solution when the   hole shrinks to a point, i.e. $\varepsilon$ goes to zero?} Unfortunately, the approach used by Ebobisse and Ahmedou does not allow to get any information about the qualitative properties of the solution they found.
  That is why we are interested in building a solution to problem \eqref{P2} using a perturbative approach, which naturally provides extremely accurate information about the asymptotic  behavior of the solution as the size of the hole converges to zero. Moreover,  the main important feature of our point of view is that it
   is the starting point to obtain a large number of positive or  sign changing solutions to \eqref{P2}
when the domain has one or more small circular holes.

In order to state our main result it is necessary to introduce the  bubble, which is the key ingredient of our  proof. A bubble is
  the function  $U_{\mu, \xi }$  defined by
\begin{equation*}
U_{\mu, \xi }\left( x\right) =\alpha_{N}\left( \frac{\mu }{\mu^{2}+\left\vert x-\xi \right\vert ^{2}}\right) ^{\frac{N-4}{2}}, \quad x,\xi\in\mathbb{R}^N,\quad \mu\in \mathbb{R}^+
\end{equation*}
where $\alpha_{N}:=(N(N-4)(N-2)(N+2))^{\frac{N-4}{8}}$. It is well known (see Lin   \cite{Chang-ShouLin}) that they are all the positive solutions in ${\mathcal D}^{2,2}(\mathbb R^N)$ of the limit problem
\begin{equation}\label{Pinf}
\Delta ^{2}U=U^{p}\quad \hbox{ in }\ \mathbb{R}^ N.\end{equation}

 To find a good approximation of the solution we are looking for, we need to project the bubble onto the domain $\Omega_\varepsilon$ with Dirichlet boundary conditions.
    For any function $u\in  {\mathcal D}^{2,2}(\mathbb R^N)$ we denote by $Pu$ its protection on $H^2(\Omega_\varepsilon)\cap H^1_0(\Omega_\varepsilon)$, i.e.   the unique solution of the problem
\begin{equation*}
\left\{
\begin{array}{ll}
\Delta ^{2}P u= \Delta ^{2}  u & \mbox{ in }\Omega _\varepsilon, \medskip\\
Pu=\Delta Pu=0 & \mbox{ on }\partial\Omega _\varepsilon.
\end{array}
\right.
\end{equation*}
Our main result reads as follows.

\begin{theorem}\label{maintheorem}
 There exists $\varepsilon_0>0$ such that for every $\varepsilon \in (0,\varepsilon_0)$
there exists a solution $u_\varepsilon$ to the problem (\ref{P2})
$$
u_\varepsilon(x)= PU_{\mu_\varepsilon,\xi_\varepsilon}(x)+ {\phi_\varepsilon}(x),\quad x\in\Omega\setminus B(\xi_0,\varepsilon),
$$
where the weight $\mu_\varepsilon$ of the bubble  satisfies
 $$\mu_\varepsilon=\varepsilon^{\frac{1}{2}(\frac{N+4}{N-3})}d_\varepsilon\ \hbox{ for some }\ d_\varepsilon\to d\in\mathbb{R}^+,$$
the center $\xi_\varepsilon$ of the bubble  satisfies
$$\xi_\varepsilon=\xi_0+\mu_\varepsilon\tau_\varepsilon\ \hbox{ for some }\ \tau_\varepsilon\to \tau\in\mathbb{R}^N$$
and the rest function $\phi_\varepsilon$ is a remainder term.
  \end{theorem}

We would like  to point out the difficulties arising in the construction of the solution $u_\varepsilon.$
The solution  $u_\varepsilon $ looks like a bubble concentrating around the center $\xi_0$ of the removed ball $\overline{B(\xi_0,\varepsilon)}$ as $\varepsilon$ goes to zero, so the point where the concentration takes place does not belong to the domain. Its profile resembles a volcano whose crater is the point $\xi_0$.
 Taking into account that the solution concentrates around the hole and at the same time it must satisfy zero Dirichlet boundary condition  on the boundary of the hole, it turns to be extremely delicate to study the behavior of the solution in the region around the hole. A key estimate is contained in Proposition \ref{cru},
 where the expansion of the projection $PU_{\mu, \xi }$ is performed.

\medskip
At this stage it is useful  to compare the Paneitz-type problem \eqref{P} with the Yamabe-type problem
\begin{equation}\label{P4}
\left\{
\begin{array}{ll}
-\Delta  u=u^{N+2\over N-2  } & \mbox{ in }\Omega  ,\medskip\\
u>0&\mbox{ in }\Omega  ,\medskip\\
u=\Delta u=0 & \mbox{ on }\partial  \Omega  .
\end{array}
\right.
\end{equation}
If the domain $\Omega$ is star shaped Pohozaev's identity \cite{Po} implies that problem \eqref{P4} has only the trivial solution, while Bahri-Coron \cite{Ba-Co} proved that (\ref{P4})  possesses a solution provided that some homology group of   $\Omega$ is nontrivial.
Again the topology assumption on the domain is not necessary for the existence of solution to \eqref{P4} as proved by Passaseo in \cite{Passaseo1,Passaseo2}.

In particular, problem \eqref{P2} is similar to the   following one
\begin{equation}\label{P3}
\left\{
\begin{array}{ll}
-\Delta  u=u^{N+2\over N-2  } & \mbox{ in }\Omega_\varepsilon:=\Omega\setminus \overline{B(\xi_0,\varepsilon) } ,\medskip\\
u>0&\mbox{ in }\Omega_\varepsilon  ,\medskip\\
u=\Delta u=0 & \mbox{ on }\partial  \Omega_\varepsilon ,
\end{array}
\right.
\end{equation}
 which was firstly studied by Coron \cite{co}, who proved the existence of a solution provided $\varepsilon$ is small enough.    If we go deep into the similarities between the two problems \eqref{P2} and \eqref{P3}, we could conjecture that all the results obtained for the Yamabe-type problem \eqref{P3} concerning existence  of positive  and/or sign changing solutions
when the domain has one or more small circular holes are also true for the Paneitz-type problem \eqref{P2}.
 In the present paper we only build solutions   which concentrate at the center of the  hole,
 whose profile is a single bubble   as in Theorem \ref{maintheorem}. We point out that, arguing as Musso-Pistoia \cite{MP2} (see also Rey \cite{R}, Lewandowski \cite{L} and Li-Yan-Yang \cite{LYY}), it is also possible to get some multiplicity results when the domain  has one or more small circular holes. Finally,     we believe that arguing exactly as in Musso-Pistoia \cite{MP1} and Ge-Musso-Pistoia   \cite{GMP} one can construct
  an arbitrary large number of   sign changing solutions whose profile is a superposition of bubbles with alternate sign which concentrate at the center of the hole.
Actually, it is worth noting that the expansion of the projection of the bubble    given in Proposition \ref{cru} is   the starting point in the construction of all these type of solutions.  It is also important to remark that  the proof in the case of problem \eqref{P2} could be extremely tedious, even if it can be carried out using the same arguments developed in the study of problem \eqref{P3}.

\medskip

 The proof of Theorem \ref{maintheorem} relies on a very well known Lyapunov-Schmidt reduction.  In particular, we will follow the arguments
used by Del Pino-Felmer-Musso \cite{DFM} and Musso-Pistoia in \cite{MP}.
We shall omit many details on the proof
because they can be found, up to some minor modifications, in  those papers.
We only compute what cannot be deduced from known results.
The paper is arranged as follows. Section  \ref{section2}  is devoted to compute the first order approximation of the solution, while Section \ref{section3}
contains the main steps of the proof of Theorem \ref{maintheorem}.

\medskip

{\em Acknowledgements}
S.A. was partially supported by Fondecyt Grant No. 11110482, USM Grant No. 121210 and Programa Basal, CMM, U. de Chile, while A.P.  was partially supported by Funds  for Cooperation between La Sapienza Universit\`a di Roma and  Pontificia Universidad Cat\'olica de Chile.

\setcounter{equation}{0}
\section{The first order approximation of the solution  }\label{section2}

Without loss of generality, we can assume that the center of the hole is the origin, i.e. $0\in\Omega$ and $\Omega_\varepsilon:=\Omega\setminus \overline B_\varepsilon $ where $B_\varepsilon:=B(0,\varepsilon).$

We look for a solution to problem \eqref{P2} as
\begin{equation}\label{ansatz}
u_\varepsilon (x):= PU_{\mu,\xi}(x)+\phi_\varepsilon(x),
\end{equation}
where the weight $\mu$  of the bubble and the center $\xi$ of the bubble  satisfy
\begin{equation}\label{setting}
\mu:=d\varepsilon^{\sigma}\quad  \hbox{ and }   \quad \xi:=\mu\tau,\quad \hbox{where}\ d\in\mathbb{R}^+\cap[\delta,1/\delta],\   \tau\in\mathbb{R}^N\cap \overline {B(0,1/\delta)}\end{equation}
 for some $\delta>0 $ and the exponent $\sigma$ is chosen so that \eqref{sigma1}, namely
 \begin{equation}\label{sigma}\sigma:={N-2\over 2(N-3)},
\end{equation}

  The rest term $\phi_\varepsilon$ is a remainder term which belongs to a suitable space, which will be introduced in the next section.

Our aim is to write the first order approximation of the solution given in  \eqref{ansatz}, namely to write the first order approximation of the bubble $PU_{\mu,\xi}$
when $\mu$ and $\xi$ satisfy \eqref{setting}.

Let $G$ be the Green's function  for the bi-Laplacian operator on $\Omega,$ that is given $x\in\Omega$
\begin{equation}\label{green}
\left\{
\begin{array}{ll}
\Delta ^{2}G\left( x,\cdot \right) =\gamma_N\delta _{x} & \mbox{ in }\Omega,\medskip\\
G\left( x,\cdot \right) =\Delta G\left( x,\cdot \right) =0 & \mbox{ on }\partial \Omega ,
\end{array}
\right.
\end{equation}
where $\gamma_N:=(N-4)(N-2)\mathrm{meas}  \left(\mathbb S^{N-1}\right).$
Let $H$ be its regular part, i.e.
\begin{equation*}
H\left( x,y\right) =\frac{1}{\left\vert x-y\right\vert ^{N-4}}-G\left(x,y\right) ,
\end{equation*}
which verifies
\begin{equation*}
\left\{
\begin{array}{ll}
\Delta ^{2}H\left( x,\cdot \right) =0 & \mbox{ in }\Omega ,\medskip\\
H\left( x,\cdot \right) =\frac{1}{\left\vert x-y\right\vert ^{N-4}} & \mbox{ on }\partial \Omega,\medskip  \\
\Delta H\left( x,\cdot \right) =-2\left( N-4\right) \frac{1}{\left\vert x-y\right\vert ^{N-2}} & \mbox{ on }\partial \Omega.
\end{array}
\right.
\end{equation*}
The function $\upsilon _{\mu,\xi  }$ defined by
\begin{equation}\label{1app}
\upsilon _{\mu,\xi  }\left( y\right) =U_{\mu, \xi  }(y) -PU_{\mu, \xi} (y),\quad y\in \Omega \setminus \overline{B}_{\varepsilon },
\end{equation}
is the unique solution of the problem
\begin{equation}\label{1app1}
\left\{
\begin{array}{ll}
\Delta ^{2}\upsilon _{\mu,\xi  }=0 & \mbox{ in }\Omega \setminus \overline{B}_{\varepsilon } ,\medskip\\
\upsilon _{\mu,\xi  }=U_{\mu, \xi } &  \mbox{ on }\partial (\Omega \setminus \overline{B}_{\varepsilon }),\medskip \\
\Delta \upsilon _{\mu, \xi  }=\Delta U_{\mu,\xi  }& \mbox{ on }\partial (\Omega \setminus\overline{B}_{\varepsilon }).
\end{array}
\right.
\end{equation}
 We introduce    the   problem
\begin{equation}\label{paux1}
\left\{\begin{array}{lll}
\Delta^2 \Upsilon=0 &\mbox{ in }\mathbb{R}^N\setminus \overline{B}_1,\medskip\\
\Upsilon=2&\mbox{ on }\partial B_1,\medskip\\
\Delta \Upsilon=-2(N-4)&\mbox{ on }\partial {B}_1,\medskip\\
\Upsilon\in \mathcal{D}^{2,2}(\mathbb{R}^N\setminus {B}_1),
\end{array}\right.
\end{equation}
which is a sort of  limit problem of \eqref{1app1} obtained by scaling by $\varepsilon.$
It is immediate to check that \eqref{paux1} has an unique solution $\Upsilon$ given by
\begin{equation}\label{defvarphi}
\Upsilon(x):=\varphi_1+\varphi_2,\quad  \hbox{where}\quad \varphi_1(x):=\frac{1}{|x|^{N-4}}\quad\mbox{and}\quad \varphi_2(x):=\frac{1}{|x|^{N-2}},\quad x\in\mathbb{R}^N\setminus \overline{B}_1.
\end{equation}
The first order approximation of the function $\upsilon_{\mu,\xi}$ defined in \eqref{1app} is given in the following.
 \begin{proposition}\label{cru}
 Set
 \begin{equation}\label{defR}
R_{\varepsilon}:=PU_{\mu,\xi}(x)-U_{\mu,\xi}(x)+\alpha_N\mu^{\frac{N-4}{2}}H(x,\xi)+a_1\varphi_1\left( \frac{x}{\varepsilon}\right)+a_2\varphi_2\left( \frac{x}{\varepsilon}\right),
\end{equation}
where
\begin{equation}\label{defa1}
a_1(\varepsilon,d,\tau):=-{\Delta U(\tau)\over 2(N-4)} {\varepsilon^ 2  \over  \mu^{N\over2}}\end{equation} and
\begin{equation}\label{defa2}a_2(\varepsilon,d,\tau):=U(\tau){1\over \mu^{N-4\over2}}+{\Delta U(\tau)\over 2(N-4)} {\varepsilon^ 2  \over  \mu^{N\over2}}.
\end{equation}

Let $\delta>0$ be fixed and assume that \eqref{setting} holds. There exists a positive constant $C>0$   such that
\begin{equation}\label{estimateR}
 |R_{\varepsilon}(x)|\leq C \left( {\varepsilon^{N-1}\over \mu^{N+2\over2}}{1\over |x|^{N-4}}+ {\varepsilon^{N-1}\over \mu^{N-2\over2}}{1\over |x|^{N-2}}\right)\quad \hbox{for any}\ x\in \Omega _\varepsilon.
\end{equation}
and
\begin{equation}\label{estimatedeltaR}
 |\Delta R_{\varepsilon}(x)|\leq C {\varepsilon^{N-1}\over \mu^{N+2\over2}}{1\over |x|^{N-2}} \quad \hbox{for any}\ x\in \Omega _\varepsilon.
\end{equation}

\end{proposition}

\begin{proof}
It is useful to   remark that
$$a_1(\varepsilon,d,\tau) = \alpha_N\frac{1}{2}\frac{\varepsilon^{2} (2|\tau|^2+N)}{\mu^{\frac{N}{2}}(1+|\tau|^2)^{\frac{N}{2}}}$$
and
$$a_2(\varepsilon,d,\tau) =
\alpha_N\left( \frac{1}{\mu^{\frac{N-4}{2}}(1+|\tau|^{2})^{\frac{N-4}{2}}}-\frac{1}{2}\frac{\varepsilon^{2}(2|\tau|^2+N)}{\mu^{\frac{N}{2}}(1+|\tau|^2)^{\frac{N}{2}}}  \right).
$$

The function $R_\varepsilon$ solves the problem
\begin{equation*}
\left\{
\begin{array}{ll}
\Delta^2 R_\varepsilon=0&\mbox{ in }\Omega\setminus \overline{B}_\varepsilon,\medskip\\
R_\varepsilon=\displaystyle \alpha_N\bigg(-\frac{\mu^{\frac{N-4}{2}}}{(\mu^2+|x-\xi|^2)^{\frac{N-4}{2}}}+
\frac{\mu^{\frac{N-4}{2}}}{|x-\xi|^{N-4}}+a_1\frac{\varepsilon^{N-4}}{|x|^{N-4}}+a_2\frac{\varepsilon^{N-2}}{|x|^{N-2}}\bigg)&\mbox{ on }\partial\Omega,\medskip\\
R_\varepsilon=\displaystyle  \alpha_N\bigg(-\frac{\mu^{\frac{N-4}{2}}}{(\mu^2+|x-\xi|^2)^{\frac{N-4}{2}}}+
\mu^{\frac{N-4}{2}}H(x,\xi)+ \frac{1}{\mu^{\frac{N-4}{2}}(1+|\tau|^2)^{\frac{N-4}{2}}}  \bigg)&\mbox{ on }\partial B_\varepsilon,\medskip\\
\Delta R_\varepsilon=\displaystyle \alpha_N(N-4)\bigg(\frac{\mu^{\frac{N-4}{2}} (2|x-\xi|^2+N\mu^2)}{(\mu^2+|x-\xi|^2)^{\frac{N}{2}}}-
\frac{2\mu^{\frac{N-4}{2}} }{|x-\xi|^{N-2}}-\frac{2|\tau|^2+N}{\mu^{\frac{N}{2}}(1+|\tau|^2)^{\frac{N}{2}}}\frac{\varepsilon^{N-2}}{|x|^{N-2}} \bigg)&\mbox{ on }\partial\Omega,\medskip\\
\Delta R_\varepsilon=\displaystyle \alpha_N (N-4)\bigg(\frac{\mu^{\frac{N-4}{2}}(2|x-\xi|^2+N\mu^2)}{(\mu^2+|x-\xi|^2)^{\frac{N}{2}}}+
\frac{\mu^{\frac{N-4}{2}} }{N-4}\Delta H(x,\xi) -\frac{2|\tau|^2+N}{\mu^{\frac{N}{2}}(1+|\tau|^2)^{\frac{N}{2}}} \bigg)&\mbox{ on }\partial B_\varepsilon.
\end{array}
\right.
\end{equation*}
Let us set $\hat{R}_\varepsilon(y):=\mu^{-\frac{N-4}{2}} R_\varepsilon (\varepsilon y),\quad y\in (\varepsilon^{-1}\Omega\setminus \overline{B}_1)$. It solves
the problem
\begin{equation*}
\left\{
\begin{array}{ll}
\Delta^2 \hat{R}_\varepsilon=0&\mbox{ in }(\varepsilon^{-1}\Omega\setminus \overline{B}_1),\medskip\\
\hat{R}_\varepsilon=\displaystyle \alpha_N\bigg(-\frac{1}{(\mu^2+|\varepsilon y-\xi|^2)^{\frac{N-4}{2}}}+
\frac{1}{|\varepsilon y-\xi|^{N-4}}+\frac{a_1}{\mu^{\frac{N-4}{2}}}\frac{1}{|y|^{N-4}}+\frac{a_2}{\mu^{\frac{N-4}{2}}}\frac{1}{|y|^{N-2}}\bigg)&\mbox{ on }\partial(\varepsilon^{-1}\Omega),\medskip\\
\hat{R}_\varepsilon=\displaystyle  \alpha_N\bigg(-\frac{1}{(\mu^2+|\varepsilon y-\xi|^2)^{\frac{N-4}{2}}}+
H(\varepsilon y,\xi)+ \frac{1}{\mu^{N-4}(1+|\tau|^2)^{\frac{N-4}{2}}}  \bigg)&\mbox{ on }\partial B_1,\medskip\\
\Delta \hat{R}_\varepsilon=\displaystyle  \alpha_N(N-4)\bigg(\frac{2|\varepsilon y-\xi|^2+N\mu^2}{(\mu^2+|\varepsilon y-\xi|^2)^{\frac{N}{2}}}-
\frac{2 }{|\varepsilon y-\xi|^{N-2}}-\frac{2|\tau|^2+N}{\mu^{N-2}(1+|\tau|^2)^{\frac{N}{2}}}\frac{1}{|y|^{N-2}} \bigg)\varepsilon^2&\mbox{ on }\partial(\varepsilon^{-1}\Omega),\medskip\\
\Delta \hat{R_\varepsilon}=\displaystyle  \alpha_N (N-4)\bigg(\frac{2|\varepsilon y-\xi|^2+N\mu^2}{(\mu^2+|\varepsilon y-\xi|^2)^{\frac{N}{2}}}+\frac{1}{N-4}\Delta H(\varepsilon y,\xi) -\frac{2|\tau|^2+N}{\mu^{N-2}(1+|\tau|^2)^{\frac{N}{2}}} \bigg) \varepsilon^2&\mbox{ on }\partial B_1.
\end{array}
\right.
\end{equation*}
Moreover, the following estimates hold true for $\hat{R}_{\varepsilon}$.
\begin{equation}\label{estimatehatRbOmega}
0\leq \hat{R}_\varepsilon (y) = O(\mu^2)\quad\forall y\in\partial (\varepsilon^{-1}\Omega),
\end{equation}
\begin{equation}\label{estimatehatRbhole}
 |\hat{R}_\varepsilon (y) | =O\left(\frac{\varepsilon }{\mu^{N-3}}\right)\quad\forall y\in \partial B_1,
\end{equation}

\begin{equation}\label{estimateDhatRbOmega}
0\leq -\Delta \hat{R}_\varepsilon (y)=  O(\varepsilon^2\mu^2)\quad\forall y\in\partial (\varepsilon^{-1}\Omega)
\end{equation}
and
\begin{equation}\label{estimateDhatRbhole}
  |\Delta \hat{R}_\varepsilon (y)| = O\left(\frac{ \varepsilon^{3 }}{\mu^{{N}-1}}\right)\quad\forall y\in\partial B_1.
 \end{equation}

Now, let $R>0$, $\mathbf{d}:=\mbox{diam}\,\Omega$ and let $\Phi$ be a solution of  the problem
\begin{equation*}
\left\{\begin{array}{rll}
\Delta^2 \Phi= 0&\mbox{in } B_{\varepsilon^{-1}\mathbf{d}}\setminus \overline{B}_1,\medskip\\
\Phi=\alpha&\mbox{on }\partial B_{\varepsilon^{-1}\mathbf{d}},\medskip\\
\Phi=\beta &\mbox{on }\partial B_1,\medskip\\
\Delta \Phi= \alpha'&\mbox{on }\partial B_{\varepsilon^{-1}\mathbf{d}},\medskip\\
\Delta \Phi= \beta'&\mbox{on }\partial B_1,
\end{array}\right.
\end{equation*}
for some arbitrary   numbers $\alpha, \beta, \alpha'$ and $\beta'$.
A straightforward computation shows that
$$\Phi(y)={A\over |y|^{N-4}}+{B\over |y|^{N-2}}+C|y|^2+D$$ is a solution to such a problem for a suitable choice of $A,B,C$ and $D.$
If we choose
$$\alpha=c \mu^2,\ \beta= c \frac{\varepsilon^{1-\sigma}}{\mu^{N-4}},\ \alpha'=c\varepsilon^2\mu^2,\ \beta'=c\frac{ \varepsilon^{3-\sigma}}{\mu^{{N}-2}}$$
for some positive constant $c,$ it is easy to check that
$$A=O\left( \frac{ \varepsilon^{3 }}{\mu^{{N}-1}}\right),\ B=O\left(\frac{\varepsilon }{\mu^{N-3}}\right),\ C=O\left(\varepsilon^2\mu^2\right),\ D=O(1).$$

Then,  by using  a comparison argument   and taking into account estimates (\ref{estimatehatRbOmega})-(\ref{estimateDhatRbhole}), we deduce
\begin{equation}\label{estimatehatR}
|\hat{R}_\varepsilon(y)|\leq   M\left(\frac{\varepsilon^3  }{\mu^{{N}-1}|y|^{N-4}} +\frac{\varepsilon  }{\mu^{N-3}|y|^{N-2}}+ \varepsilon^2\mu^2|y|^2+ \varepsilon  \right)\quad \forall y\in \Omega_\varepsilon,
\end{equation}
and
\begin{equation}\label{estimatedeltahatR}
|\Delta\hat{R}_\varepsilon(y)|\leq   M\left(\frac{\varepsilon^3  }{\mu^{{N}-1}|y|^{N-2}} + \varepsilon^2\mu^2 \right)\quad \forall y\in \Omega_\varepsilon,
\end{equation}
for some  positive constant    $M$, which is  independent  on $\varepsilon$. Therefore,  by     (\ref{estimatehatR}) we deduce  (\ref{estimateR})
and by    (\ref{estimatedeltahatR}), we deduce  (\ref{estimatedeltaR}).
\end{proof}

\setcounter{equation}{0}
\section{Scheme of the proof}\label{section3}

  As we said, the proof of Theorem \ref{maintheorem} relies on a very well known
Lyapunov-Schmidt reduction. In this section we will sketch the main steps of the proof skipping many details, because
  the arguments used in the proof are very similar to the ones used in \cite{DFM,MP}.
  The unique new computation is the estimate \eqref{estimateJ0} of the reduced energy.

 \medskip
First of all, it is useful to perform a change of variables.  Let  us consider the expanded domain
$
\tilde \Omega_\varepsilon:=\varepsilon^{-\sigma} \Omega _\varepsilon
$
and $y=\varepsilon^{-\sigma}x\in\tilde\Omega_\varepsilon$, $x\in\Omega _\varepsilon$, where $\sigma$ is defined in \eqref{sigma}.
Therefore, $u$ solves the problem (\ref{P}) if and only if the function
\begin{equation}\label{defve}
 v_{\varepsilon}(y)=\varepsilon^{\sigma{N-4\over2}}\, u(\varepsilon^{\sigma}y),\quad y\in\tilde\Omega_{\varepsilon}
 \end{equation}
 solves the problem
\begin{equation}\label{Pve}
\left\{
\begin{array}{rlll}
\Delta^2 v\!\!\!&\!= f( v) & {\mbox{ in }} \tilde\Omega_{\varepsilon},\medskip\\
 v\!\!\!&\!= \Delta v=0 & {\mbox{ on }} \partial \tilde\Omega_{\varepsilon}.
\end{array}
\right.
\end{equation}
Here $f(v)=(v^+)^p$ and
 $
p:=\frac{N+4}{N-4}.
 $
 We point out that the operator $\Delta^2$ with Dirichlet boundary condition satisfies the maximum principle. So any solution to \eqref{Pve} is a positive function.

In the expanded variables, the  solution  we are looking  for   looks like
\begin{equation}\label{solutionv}
v(y)=V(y)+\tilde \phi(y),\  y\in\tilde\Omega_\varepsilon,\quad
V(y):=\varepsilon^{\sigma{N-4\over2}} PU_{\mu,\xi}(\varepsilon^{\sigma}y)\ \hbox{ and } \ \tilde\phi(y):=\varepsilon^{\sigma{N-4\over2}} {\phi}(\varepsilon^{\sigma} y).
\end{equation}
 It is important to point out that the function $V$ is nothing but the projection onto $\mathrm H^2(\tilde\Omega_\varepsilon)\cap\mathrm H^1_0(\tilde\Omega_\varepsilon)$
 of the function $\varepsilon^{\sigma{N-4\over2}}  U_{\mu,\xi}(\varepsilon^{\sigma}y)= U_{d,\xi'}( y)$ where we denote by $\xi'$ the point $\varepsilon^{-\sigma}\xi $
 and by $d$ the number $\varepsilon^{-\sigma}\mu$ (see \eqref{setting}).

 In terms of $\tilde\phi$ problem \eqref{Pve} rewrites as

 \begin{equation}\label{Pve2}
\left\{
\begin{array}{rlll}
L  (\tilde \phi)\!\!\!&\!=  N (\tilde \phi)+E  & {\mbox{ in }} \tilde\Omega_{\varepsilon},\medskip\\
\tilde \phi\!\!\!&\!= \Delta \tilde \phi=0 & {\mbox{ on }} \partial \tilde\Omega_{\varepsilon}.
\end{array}
\right.
\end{equation}

 where the linear operator $L$ is defined by
 $$L  (\tilde \phi):=\Delta^2  \tilde \phi-f'(V)\tilde \phi, $$
 the second order term $N(\phi)$ is defined by
 $$ N (\tilde \phi):= f\left(V +\tilde \phi\right) -  f(V  )-f'(V)\tilde \phi $$
and the error term $E$ is defined by
\begin{equation}E :=   f(V)-f(U_{d,\xi'}). \label{eeee}
 \end{equation}

 To prove our result we follow the usual strategy of the Lyapunov-Schmidt procedure.

\bigskip
\emph{ Step 1. We solve a nonlinear problem. }

More precisely,
 given $d>0$ and $\tau\in\mathbb R^N$ (see \eqref{setting}), we find a function $\tilde\phi=\tilde\phi(\varepsilon,d,\tau)$ such that for some real numbers $c_i$'s
 \begin{equation}\label{Pve3}
\left\{
\begin{array}{rlll}
L  (\tilde \phi)\!\!\!&\!=  N (\tilde \phi)+E+\sum\limits_{i=0}^Nc_if'(V)W_i  & {\mbox{ in }} \tilde\Omega_{\varepsilon},\medskip\\
\tilde \phi\!\!\!&\!= \Delta \tilde \phi=0 & {\mbox{ on }} \partial \tilde\Omega_{\varepsilon} ,\medskip\\
\int\limits_{\tilde\Omega_\varepsilon}\tilde \phi(y)f'(V)(y)W_i(y)dy\!\!\!&\!=0& {\mbox{ for any }}  i=0,1,\dots,N.
\end{array}
\right.
\end{equation}
 The  functions $W_i$ are defined as follows. It is known (see \cite{LuWei}) that the set of the solutions to the linearized equation
 $$ \Delta^2 \vartheta - f'( {U}_{\mu,\xi})\vartheta=0\ \hbox{ in}\ {\mathbb{R}}^N,\quad \vartheta\in\mathcal D^{2,2}(\mathbb R^N)
$$
is a $(N+1)-$dimensional linear space spanned by the functions
\begin{equation*}
Z_0(x):=\frac{\partial U_{\mu,\xi}}{\partial \mu}(x)=\alpha_N \left(\frac{N-4}{2}\right)\mu^{\frac{N-6}{2}} \frac{|x- \xi|^2-\mu^2}{(\mu^2+|x-\xi|^2)^{\frac{N-2}{2}}},
\end{equation*}
\begin{equation*}
\displaystyle Z_i(x):=\frac{\partial U_{\mu,\xi}}{\partial \xi_i}(x)=\alpha_N(N-4) {\mu}^{\frac{N-4}{2}} \frac{ x_i-\xi_i}{(\mu^2+|x-\xi|^2)^{\frac{N-2}{2}}},\quad i=1,2,\dots ,N.
\end{equation*}
We denote by  $PZ_i$ the projection of $Z_i$ onto   $\mathrm H^1_0(\Omega _\varepsilon)\cap \mathrm H^2(\Omega\ _\varepsilon)$  and we set
$$
W_i(y):=\varepsilon^{\sigma{N-4\over2}}PZ_i(\varepsilon^{\sigma}y),\quad y\in\tilde \Omega_\varepsilon\quad i=0,1,2,\ldots,N.
$$

 In order to solve problem \eqref{Pve3} it is necessary to study the linear problem naturally associated to it. More precisely,
  given $d>0$ and $\tau\in\mathbb R^N$ (see \eqref{setting}) and a function $h\in C^0 (\overline{\tilde \Omega _\varepsilon} )$, find a function $\tilde\phi $ such that   for some real numbers $c_i$'s
 \begin{equation}\label{Pve4}
\left\{
\begin{array}{rlll}
L  (\tilde \phi)\!\!\!&\!=  h+\sum\limits_{i=0}^Nc_if'(V)W_i  & {\mbox{ in }} \tilde\Omega_{\varepsilon},\medskip\\
\tilde \phi\!\!\!&\!= \Delta \tilde \phi=0 & {\mbox{ on }} \partial \tilde\Omega_{\varepsilon} ,\medskip\\
\int\limits_{\tilde\Omega_\varepsilon}\tilde \phi(y)f'(V)(y)W_i(y)dy\!\!\!&\!=0& {\mbox{ for any }}  i=0,1,\dots,N.
\end{array}
\right.
\end{equation}

 To study the invertibility of the linear operator $L$ we introduce the $\mathrm L^\infty-$weighted spaces $\mathrm L^\infty_*(\tilde\Omega_\varepsilon)$
 and $\mathrm L^\infty_{**}(\tilde\Omega_\varepsilon)$ to be, respectively, the spaces of functions defined on $\tilde\Omega_\varepsilon$
 with finite $\|\cdot\|_*$ and $\|\cdot\|_{**}$  norms defined by
 $$\|\eta\|_*=\sup\limits_{y\in\tilde\Omega_\varepsilon}\sum\limits_{i=0}^3\left|(1+|y-\xi'|^2)^{2+i\over 2}\sum\limits_{|\alpha|=i}D^\alpha\eta(y)\right|$$
and
$$\|\eta\|_{**}=\sup\limits_{y\in\tilde\Omega_\varepsilon}\left|(1+|y-\xi'|^2)^4\eta(y)\right|.$$

The operator $L$ is uniformly invertible with respect to the above weighted norms provided $\varepsilon$ is small enough as it is proved in the next result.
\begin{proposition}\label{p1}
Let $\delta>0$ be fixed and assume that (\ref{setting}) holds true. Then there exist constants $\varepsilon_0 >0$ and  $C>0$, such that  for every $0<\varepsilon<\varepsilon_0$ and $h\in C^0 (\overline{\tilde \Omega _\varepsilon} )$,   problem (\ref{Pve4}) admits a unique solution $
T_\varepsilon (d,{\xi'},h)$. Furthermore, the map $(d,{\xi'})\mapsto   T_{\varepsilon}(d,{\xi'},h)$ is of class  $C^1$ for the $\|\cdot\|_*-$norm and satisfies
$$
\| T_{\varepsilon}(d,{\xi'},h)\|_* \leq C \| h \|_{**} ,\qquad \| \nabla_{(d,{\xi'})} T_{\varepsilon}(d,{\xi'},h)\|_{*}\leq C \|h\|_{**}.
$$
Moreover,
\begin{equation}\label{ciref}
|c_i|<C\|h\|_{**}\quad \forall i.
\end{equation}
\end{proposition}
\begin{proof}
We argue exactly as in Section 5 of \cite{MP}.
\end{proof}

Finally, we have all the ingredients to solve problem \eqref{Pve}.

\begin{proposition}\label{p3}
Let $\delta>0$ be fixed and assume that (\ref{setting}) holds true. Then there exist constants $\varepsilon_0 >0$ and  $C>0$ such that  for every $0<\varepsilon<\varepsilon_0$  there exists a unique solution $\tilde\phi=\tilde\phi( d,\xi')$ to   problem (\ref{Pve3})  such that   the map $(d,{\xi'} )\mapsto  \tilde\phi(d,\xi')$ is of class  $C^1$ for the $\|\cdot\|_*-$norm and
$$
 \|\tilde\phi \|_{*} \leq C   \varepsilon^{(N-2)(N-4)\over   2(N-3) }  ,\qquad
 \| \nabla_{(d,\xi')} \tilde\phi \|_{*} \le C  \varepsilon^{(N-2)(N-4)\over 2( N-3) } .
$$
\end{proposition}
\begin{proof}
We argue exactly as in Section 6 of \cite{MP}, once one has the estimate of the error term   $E$    defined   in (\ref{eeee}).
Indeed by Proposition \ref{cru} we deduce
 $\|E\|_{**}=O ( \varepsilon^{(N-2)(N-4)\over   2(N-3) })$.
 \end{proof}

\bigskip
\emph{ Step 2. We reduce the problem to a finite dimensional one. }

 Let us consider the function $J_\varepsilon:\mathbb R^+\times\mathbb R^N\to\mathbb R$ defined by
\begin{equation}\label{je}
J_\varepsilon(d,\tau):=I_\varepsilon(V+\tilde\phi ),\end{equation}
where $\phi $ is the function found in Proposition \ref{p3} and the functional $I_\varepsilon:H^{2}(\tilde\Omega_\varepsilon)\cap H^1_0(\tilde\Omega_\varepsilon)\to\mathbb R$ is defined by
$$
I_\varepsilon(v):=\frac{1}{2}\int_{\tilde \Omega_\varepsilon} |\Delta v|^2dy-\frac{1}{p+1 } \int_{\tilde\Omega_\varepsilon}(v^+)^{p+1 }dy.
$$

\begin{proposition}
\label{rido}
\begin{itemize}
\item[(i)]
The function $v=V+\tilde\phi$ is a solution to problem \eqref{Pve2}, namely $c_i=0$ in \eqref{Pve3} for all $i$'s, if and only if $(d,\xi)$ is a critical point of $J_\varepsilon.$

\item[(ii)]
It holds true that
\begin{equation*}\label{estimateJ0}
J_{\varepsilon }(d,\tau)=a_N+\varepsilon^{(N-2)(N-4)\over 2( N-3)}\Psi(d,\tau)+o\left(\varepsilon^{(N-2)(N-4)\over 2(N-3)}\right),
\end{equation*}
$C^1-$uniformly with respect to $(d,\tau)$ in compact sets of $\mathbb R^+\times \mathbb R^n.$
Here
\begin{equation}\label{PSI}
\Psi(d,\tau):= -b_n\Delta U(\tau)U(\tau)\frac{1}{d^{N-2}}+c_N H(0,0) d^{N-4},\end{equation}
where the   $a_N,$ $b_N$ and $c_N$ are  positive constants defined by

\begin{align*}
&a_N:=   \int_{\mathbb{R}^N} U^{2N\over N-4}(y)dy, \ b_N:={3\over4}(N-2)\mathrm{meas} \left(\mathbb S^{N-1}\right)\
\  c_N:={1\over2}\alpha_N  \int_{\mathbb{R}^N}U^{\frac{N+4}{N-4}}(y)dy.
\end{align*}
\end{itemize}

\end{proposition}

\begin{proof}
We argue exactly as in Section 4 of \cite{MP}. We only need to compute the leading term in  the expansion of $J_{\varepsilon }(d,\xi)$, which  is nothing but the energy of the bubble, namely  $I_\varepsilon(V).  $
So we have to compute:

\begin{equation}\label{i1}
\begin{array}{llll}
I_\varepsilon(V)&\!\!\!= &\!\!\! \displaystyle\frac{1}{2} \int_{\tilde\Omega  _\varepsilon} |\Delta V|^{2}-\frac{1}{p+1} \int_{\tilde\Omega  _\varepsilon} V ^{p+1}\medskip\\
&\!\!\!=&\!\!\! \displaystyle\frac{1}{2} \int_{\Omega  _\varepsilon} |\Delta PU_{\mu,\xi}|^{2}-\frac{1}{p+1} \int_{\Omega  _\varepsilon} (PU_{\mu,\xi})^{p+1}\medskip\\
&\!\!\!=&\!\!\!\displaystyle\frac{1}{2} \int_{\Omega _\varepsilon}  U_{\mu,\xi}^{p}PU_{\mu,\xi}-\frac{1}{p+1} \int_{\Omega  _\varepsilon} (PU_{\mu,\xi})^{p+1}\medskip \\
&\!\!\!=&\!\!\!\displaystyle \frac{2}{N}\int_{\Omega  _{\varepsilon}} U_{\mu,\xi}^{p+1}
-\frac{1}{2} \int_{\Omega _{\varepsilon}} U_{\mu,\xi}^{p}(PU_{\mu,\xi}-U_{\mu,\xi})\medskip\\
&\!\!\! &\!\!\! \displaystyle-{\frac{1}{2}p}\int_{\Omega _{\varepsilon}} (t PU_{\mu,\xi}+(1-t)U_{\mu,\xi})^{p-1}(PU_{\mu,\xi}-U_{\mu,\xi})^2.
\end{array}
\end{equation}
 The first term in the R.H.S. of \eqref{i1} is estimated as follows.
\begin{equation}\label{i2}
\begin{array}{llll}
\displaystyle \int_{\Omega\setminus \overline{B}_{\varepsilon}} U_{\mu,\xi}^{p+1}
&\!\!\!=&\!\!\!\displaystyle\alpha_N^{\frac{2N}{N-4}}\int_{\Omega \setminus \overline{B}_\varepsilon} \frac{\mu^{N}}{(\mu^2+|x-\xi|^2)^N}dx \medskip\\
&\!\!\!=&\!\!\!\displaystyle\alpha_N^{\frac{2N}{N-4}}\int_{\mu^{-1}(\Omega \setminus \overline{B}_\varepsilon)} \frac{1}{(1+|y-\tau|^2)^N}dy \medskip\\
&\!\!\!=&\!\!\!\displaystyle\alpha_N^{\frac{2N}{N-4}}\int_{\mathbb{R}^N} \frac{1}{(1+|y|^2)^N}dy+O\left(\left(\frac{\varepsilon}{\mu}\right) ^{N}+\mu^N\right).
\end{array}
\end{equation}
The second term in the R.H.S. of \eqref{i1} is estimated as follows.
By Proposition \ref{cru} we get
\begin{equation}\label{i3}
\begin{array}{llll}
&\!\!\! &\!\!\!\displaystyle \int_{\Omega\setminus \overline{B}_{\varepsilon}} U_{\mu,\xi}^{p}(PU_{\mu,\xi}-U_{\mu,\xi})
=\displaystyle\int_{\Omega\setminus \overline{B}_{\varepsilon}} U_{\mu,\xi}^{p}R_{\varepsilon}\medskip\\
&\!\!\!&\!\!\! \displaystyle -
\int_{{\Omega \setminus \overline{B}_\varepsilon}} U_{\mu,\xi}^{p}\left(\alpha_N\mu^{\frac{N-4}{2}}H(x,\xi)
+a_1\varphi_1\left(\frac{x}{\varepsilon}\right)+a_2\varphi_2\left(\frac{x}{\varepsilon}\right)   \right) dx,
\end{array}
\end{equation}
where $\varphi_1$ and $\varphi_2$ are the functions defined in (\ref{defvarphi}), and
$a_1$ and $a_2$ are the constants defined in (\ref{defa1}) and \eqref{defa2}. We estimate each summand in the right hand side of \eqref{i3}. We scale $x-\xi=\mu y$ and we get
\begin{equation}\label{i4}
\begin{array}{llll}
&\!\!\! &\!\!\!\displaystyle \int_{{\Omega \setminus \overline{B}_\varepsilon}} U_{\mu,\xi}^{p}\alpha_N\mu^{\frac{N-4}{2}}H(x,\xi)dx\medskip\\
&\!\!\!=&\!\!\!\displaystyle  \alpha_N^{\frac{2N}{N-4}} \int_{\mu^{-1}(\Omega \setminus (\overline{B}_\varepsilon-\xi))}\mu^{N-4}H(\mu y+\xi,\xi)\frac{1}{(1+|y|^2)^{\frac{N+4}{2}}}dy\medskip\\
&\!\!\!=&\!\!\!\alpha_N^{\frac{2N}{N-4}}\displaystyle \mu^{N-4}H(0,0) \left( \int_{\mathbb{R}^N}\frac{1}{(1+|y|^2)^{\frac{N+4}{2}}}dy +o(1)\right),
\end{array}
\end{equation}

\begin{align}\label{i5}
 & \int_{{\Omega \setminus \overline{B}_\varepsilon}} U_{\mu,\xi}^{p}a_1\varphi_1\left(\frac{x}{\varepsilon}\right) dx\nonumber\\
& = - \frac{\Delta U(\tau)}{2(N-4)}\frac{\varepsilon^{2} }{\mu^{\frac{N}2}}\mu^{N-4\over2}\displaystyle  \int_{\mu^{-1}(\Omega \setminus (\overline{B}_\varepsilon-\xi))}\varphi_1\left(\frac{\mu}{\varepsilon}(y+\tau)\right)U^p(y)dy\nonumber\\
& = \left(\frac{\varepsilon}{\mu}\right)^{N-2}\left(- \frac{\Delta U(\tau)}{2(N-4)}\int_{\mathbb{R}^N}\frac{1}{|y+\tau|^{N-4}}U^p(y)dy+o(1)\right) \nonumber\\
& = \left(\frac{\varepsilon}{\mu}\right)^{N-2}\left(-\frac{N-2}{2 } \mathrm{meas}  (\mathbb S^{N-1}) \Delta U(\tau)U(\tau)+o(1)\right),
\end{align}
because the function $U$ solves \eqref{Pinf} and the Green's function of $\Delta^2$ in $\mathcal D^{2,2}(\mathbb R^N)$ is $1\over|x-y|^{N-4}$ with the normalization constant  given by $(N-2)(N-4)\mathrm{meas}(\mathbb S^{N-1}) $ (see also \eqref{green}),

\begin{align}\label{i6}
 &\int_{{\Omega \setminus \overline{B}_\varepsilon}} U_{\mu,\xi}^{p}a_{2}\varphi_2\left(\frac{x}{\varepsilon}\right) dx\nonumber\\
& = U(\tau)\int_{\mu^{-1}(\Omega \setminus (\overline{B}_\varepsilon-\xi))}\varphi_2\left(\frac{\mu}{\varepsilon}(y+\tau)\right)U^p(y)dy\nonumber\\
& +\frac{\Delta U(\tau)}{2(N-4)} \frac{\varepsilon^{2} }{\mu^{\frac{N}2}  } \mu^{N-4\over2}   \int_{\mu^{-1}(\Omega \setminus (\overline{B}_\varepsilon-\xi))}\varphi_2\left(\frac{\mu}{\varepsilon}(y+\tau)\right)U^p(y)dy\nonumber\\
& =  \left(\frac{\varepsilon}{\mu}\right)^{N-2}\left(U(\tau)\int_{\mathbb{R}^N}\frac{1}{|y+\tau|^{N-2}}
U^p(y)dy+o(1)\right)\nonumber\\
& + \left(\frac{\varepsilon}{\mu}\right)^{N}\left(\frac{\Delta U(\tau)}{2(N-4)}\int_{\mathbb{R}^N}\frac{1}{|y+\tau|^{N-2}}U^p(y)dy+o(1)\right),
\nonumber\\
& =  \left(\frac{\varepsilon}{\mu}\right)^{N-2}\left(U(\tau)\int_{\mathbb{R}^N}\frac{1}{|y+\tau|^{N-2}}
U^p(y)dy+o(1)\right)\nonumber\\
& =  \left(\frac{\varepsilon}{\mu}\right)^{N-2}\left(-(N-2)\mathrm{meas}(\mathbb S^{N-1}) U(\tau)\Delta U(\tau)+o(1)\right)
,\end{align}
because the function
$W=\Delta U$ solves the problem $\Delta W=U^p$ in $\mathbb R^N$ and the Green's function of $-\Delta$ in $\mathcal D^{2,2}(\mathbb R^N)$ is $1\over|x-y|^{N-2}$ with the normalization constant  given by   $(N-2)\mathrm{meas}(\mathbb S^{N-1}).$

Moreover, by \eqref{estimateR} we deduce
 \begin{equation}\label{i7}
\begin{array}{llll}
&\!\!\! &\!\!\!\displaystyle\int_{{\Omega \setminus \overline{B}_\varepsilon}} U_{\mu,\xi}^{p}|R_\varepsilon(x)| dx\medskip\\
&\!\!\!=&\!\!\!
 O\left(\displaystyle\int_{{\Omega \setminus \overline{B}_\varepsilon}} {\mu^{N+4\over2}\over(\mu^2+|x-\xi|^2)^{N+4\over2}}\left( {\varepsilon^{N-1}\over \mu^{N+2\over2}}{1\over |x|^{N-4}}+ {\varepsilon^{N-1}\over \mu^{N-2\over2}}{1\over |x|^{N-2}}\right)
 dx\right)
 \medskip\\
&\!\!\!= &\!\!\!
 O\left(\displaystyle{\varepsilon^{N-1}\over \mu^{N-2}} \int_{\mathbb R^N} {1\over(1+|y|^2)^{N+4\over2}} {1\over |y|^{N-4}}dy+\displaystyle{\varepsilon^{N-1}\over \mu^{N-1}} \int_{\mathbb R^N} {1\over(1+|y|^2)^{N+4\over2}} {1\over |y|^{N-2}}dy
\right)\medskip\\
&\!\!\!= &\!\!\!
o\left(\displaystyle{\varepsilon^{N-2}\over \mu^{N-2}} \right)
 \end{array}
\end{equation}

The last term in the R.H.S. of \eqref{i1} can be estimated as \eqref{i7}
and so
 \begin{equation}\label{i8}
{\frac{1}{2}p}\int_{\Omega _{\varepsilon}} (t PU_{\mu,\xi}+(1-t)U_{\mu,\xi})^{p-1}(PU_{\mu,\xi}-U_{\mu,\xi})^2=o\left({\varepsilon^{N-2}\over \mu^{N-2}} \right).
\end{equation}

We  collect all the estimates \eqref{i1}--\eqref{i8} and the claim follows, provided $\sigma$ is chosen so that
\begin{equation}\label{sigma1}
\mu^{N-4}\sim\left({\varepsilon\over\mu}\right)^{N-2}\quad\Rightarrow\quad \mu^{2(N-3)}\sim\varepsilon^{N-2}\quad\Rightarrow\quad \sigma={N-2\over 2(N-3)}.\end{equation}
 \end{proof}

\begin{proof}[Proof of the Theorem \ref{maintheorem}]
We know that $PU_{\mu,\xi}+\phi$ is a solution to problem \eqref{P2} if and only if the function $V+\tilde\phi$ is a solution of \eqref{Pve}.
From (i) of Proposition \ref{rido} we have that the function $V+\tilde\phi$ is a solution of \eqref{Pve} or \eqref{Pve2} if and only if $(d,\tau)$ is a critical point of the reduced energy $J_\varepsilon$ defined in \eqref{je}. Then, from (ii) of Proposition \ref{rido}, we only need to find a critical point of the function $\Psi$ defined in (\ref{PSI}), which is stable under $C^1$ perturbations. Indeed, it is easy to check that   the function $\Psi$ has a nondegenerate critical point
$\left(-{(N-2) b_N \Delta U(0)U(0)\over(N-4) c_NH(0,0)},0\right)$ of ``saddle'' type
, which is stable with respect to $C^1$ perturbations. That proves our claim.
 \end{proof}

\end{document}